\documentclass[12pt]{amsart}
\pdfoutput=1

\usepackage[T1]{fontenc}
\usepackage{lmodern}
\usepackage[kerning, spacing, tracking]{microtype}
\usepackage[utf8]{inputenc}

\usepackage[backend=biber,
            style=numams, 
            sorting=nyt, 
            isbn=false,  
            backref=true, 
            natbib=true%
           ]{biblatex}
\usepackage{amsmath, amsthm, amssymb}
\usepackage{enumitem}

\usepackage[colorlinks, 
            linkcolor=blue, 
            citecolor=blue, 
            pdfa%
           ]{hyperref}

\usepackage[hcentering, hscale=0.68, vscale=0.77]{geometry}

\swapnumbers
\newtheorem{thm}{Theorem}[section]
\newtheorem{cor}[thm]{Corollary}
\newtheorem{lem}[thm]{Lemma}
\newtheorem{prop}[thm]{Proposition}

\theoremstyle{definition}
\newtheorem{defi}[thm]{Definition}
\newtheorem*{notat}{Notation}
\newtheorem{example}[thm]{Example}

\newtheorem{remark}[thm]{Remark}

\swapnumbers
\newtheorem{step}{Step}

\newlist{enumthm}{enumerate}{1}  
\setlist[enumthm,1]{label=\textup{(\roman*)}}

\renewcommand{\geq}{\geqslant}
\renewcommand{\leq}{\leqslant}

\newcommand{\abs}[1]{\lvert #1 \rvert}
\newcommand{\erz}[1]{\langle #1 \rangle}
\newcommand{\nteq}{\unlhd} 
\newcommand{\nt}{\lhd}

\DeclareMathOperator{\Z}{\mathbf{Z}}
\DeclareMathOperator{\C}{\mathbf{C}}
\DeclareMathOperator{\N}{\mathbf{N}}
\DeclareMathOperator{\irr}{Irr}
\DeclareMathOperator{\lin}{Lin}
\DeclareMathOperator{\aut}{Aut}
\DeclareMathOperator{\syl}{Syl}

\hypersetup{pdfauthor={Frieder Ladisch},
            pdftitle = {Groups with anticentral elements},
            pdfkeywords={Finite groups, Centralizers,
                        solvable groups, Carter subgroup.}
           }

\begin{document}
\author{Frieder Ladisch}
\address{Universit\"{a}t Rostock,
         Institut f\"{u}r Mathematik,
         18051 Rostock,
         Germany}
\email{frieder.ladisch@uni-rostock.de}
\title{Groups with Anticentral Elements}
\begin{abstract}
  We study finite groups $G$ with elements $g$ such that
  $\abs{\C_G(g)} = \abs{G:G'}$. (Such elements generalize
  fixed-point-free automorphisms of finite groups.)
  We show that these groups have  a unique conjugacy class of
  nilpotent supplements for the  commutator subgroup and,
  using the classification of finite simple  groups, that these
  groups  are solvable.
\end{abstract}
        
\maketitle

\section{Definition}
Throughout the paper, let $G$ be a finite group.
\begin{prop}
  Let $a\in G$. The following assertions are equivalent:
  \begin{enumthm}
  \item $\abs{\C_G(a)}= \abs{G:G'}$ .
  \item $a^G = aG'$.
  \item $G'= [a,G]$ as set.
  \item Every nonlinear irreducible character vanishes at $a$.
  \end{enumthm}
\end{prop}
\begin{proof}
  The conjugacy class of $a$ is contained in
  $aG'$, its cardinality is $\abs{G:\C_G(a)}$.
  This shows the equivalence of the first three conditions.
  The remaining equivalence follows  from the second orthogonality
  relation:
  \[ \abs{\C_G(a)} = \sum_{\chi \in \irr (G)} \abs{\chi(a)}^2
                \geq \sum_{\lambda \in \lin(G)} \abs{\lambda(a)}^2
                   = \abs{G:G'}\]
  with equality if and only if every nonlinear irreducible
  character vanishes at $a$.
\end{proof}
\begin{defi}
  An element $a \in G$ is called \emph{anticentral (in $G$)} if
  $\abs{\C_G(a)} = \abs{G:G'}$.
\end{defi}
We see that an anticentral element has a centralizer of  minimal
possible size. We make some trivial observations:
If $G$ is nonabelian and $a$ is anticentral,
then $a\notin G'$, since $G'$ contains more
than one conjugacy class.
If $a$ is anticentral in  $G$, then
$aN$ is anticentral in $G/N$ for any $N\nteq G$.
If $G$ and $H$ are two groups, then
$(a,b)$ is anticentral in $G \times H$
if and only if $a$ is anticentral in $G$
and $b$ is anticentral in $H$.

\section{Examples}
We collect some examples of groups with anticentral elements.
The proofs are elementary in most cases.
\begin{example}
   \label{ab}In an abelian group, every element is anticentral.
\end{example}
\begin{example}
  Let $K$ be a group admitting a fixed point free
  automorphism $\alpha$.
  Then $\alpha$ is anticentral in the semidirect  product
  $G= \erz{\alpha}K$.
  More generally, take an abelian group $A$ with a surjective
  homomorphism onto $\erz{\alpha}$ and let $G$ the semidirect
  product $G=AK$.
\end{example}
\begin{example} \label{es}
  In an extraspecial $p$-group, all the elements outside the
  commutator subgroup are anticentral.
\end{example}
Parts of the following example are well
known~\cite[Lemma~12.3]{isa76}:
\begin{example}
  Let $G$ be solvable and suppose that $G'$ is a minimal
  normal subgroup of $G$. Then $G$ contains anticentral
  elements. There are two possible cases:
  \begin{enumthm}
  \item $G'\leq \Z(G)$, $\abs{G'}=p$, $G/\Z(G)$ is an
       elementary abelian $p$-group and every noncentral element
       is anticentral,
  \item $G'\cap \Z(G) = 1$, then $G/\Z(G)$ is a Frobenius
        group with kernel $(G' \times \Z(G) )/\Z(G)$ and cyclic
        complement.
        For  every $g$ not in $\C_G(G')$, we have
        $G= \C_G(g)G'$ and $\C_G(g)\cap G'=1$,
        so all these elements are anticentral.
  \end{enumthm}
  In both cases, all nonlinear characters have the same
  degree.
\end{example}
\begin{proof}
  If $G' \leq \Z(G)$ then by minimality $\abs{G'}=p$, a prime,
  and for every $g\in G$ the map $x \mapsto [x,g]$
  is actually a group homomorphism $G\to G'$ with kernel $\C_G(g)$.
  Via that map, the
  factor group $G/\C_G(g)$ embeds into $G'$ which has prime order.
  So if $g$ is noncentral, then $G/\C_G(g)\cong G'$ and
  $g$ is anticentral.
  For $x\in G$ we see $x^p\in \C_G(g)$ for every $g\in G$.

  Now suppose $G' \cap \Z(G) = 1 $ and let $1 \neq x \in G'$.
  Since $G'$ is elementary
  abelian, we see that $G' \leq \C_G(x) \nt G$,
  and as $G' = \erz{ x^g \mid g \in G}$,
  we get
  \[ \C_G(G') = \bigcap_{g\in G} \C_G(x^g) = \C_G(x). \]
  Thus if $g \notin \C_G(G')$, then
  $\C_G(g) \cap G' = 1$ and $\C_G(g)$ is abelian.
  As $\abs{\C_G(g)}\geq \abs{G:G'}$, we
  have
  $\abs{\C_G(g)G'} = \abs{\C_G(g)}\abs{G'}
       \geq \abs{G:G'}\abs{G'} =\abs{G}$.
  We conclude $G= \C_G(g)G'$.
  Then $\Z(G) = \C_G(g) \cap \C_G(G')$ and every element of
  $\C_G(g) \setminus \C_G(G')$ operates fixed point freely on
  $G'$, where $\C_G(G')= \Z(G)\times G'$.
  Thus $G/\Z(G)$ is  a Frobenius group as claimed.
  The statement on the degrees is easy to prove and well
  known~\cite{isa76}.
\end{proof}
  In his book~\cite[Lemma~12.3]{isa76}, Isaacs investigates
  solvable groups such
  that $G'$ is the unique minimal normal subgroup of $G$
  (the uniqueness is, however, not very essential).
  The result is used in the
  investigation of groups having only two character degrees.
  Instead of saying  ``A group has only two different character
  degrees'' one can say as well
   ``Every nonlinear character takes on the same value $m$ at
  $1$''.
  This suggests a natural generalization by
  replacing `$1$' in the statement by some nonidentity element
  of the  group.
  For $m \neq 0$, such groups were studied by
  Bianchi, Chillag and Pacifici~\cite{bcp06}.
  In this case,
  $G' \setminus \{1\}$ is a conjugacy class of $G$
  (so $G'$ is a minimal normal subgroup and $G$ is solvable) and
  the structure of such a group is very restricted by the
  hypothesis.
  What we do here is to study the case  $m=0$ which is
  more difficult as can be already seen from the examples
  mentioned.
\begin{example} \label{metaab}
  If $A \nt G$ such that $G/A$ is cyclic and $A$ abelian,
  then every element $a$ such that $G/A = \erz{aA}$
  is anticentral:
  Since every character of $A$ is extendible to its
  inertia group, the nonlinear characters are induced from
  proper subgroups containing $A$. Since these subgroups
  are necessarily normal and don't contain $a$,
  nonlinear characters vanish on $a$.
\end{example}
\begin{example}
  Any group of order $p^n$ where $p$ is a prime and $n\leq 4$
  contains anticentral elements. If $n\leq 3$, this follows
  from \ref{ab} and \ref{es}, so let $G$ be a group of order
  $p^4$ and assume  $G$ contains no  anticentral elements.
  Then for any $g\in G$, we have
  $\abs{\C_G(g)} > \abs{G:G'}\geq p^2$.
  Since $G$ has to be nonabelian, there  is an element $g$
  with $\abs{\C_G(g)} = p^3$, so $\C_G(g)\nt G$.
  Then $\C_G(g)$ is nonabelian, since otherwise the elements
  outside $\C_G(g)$ were anticentral by~\ref{metaab}.
  But then  $\abs{\Z( \C_G(g))}=p$,
  and we get $g \in \Z( \C_G(g))\leq \Z(G)$, contradiction.
\end{example}
\begin{example}
  With  help of the Small Groups library of GAP~\cite{gap4},
  we can see
  that there are groups containing no anticentral element of order
  $p^5$ if $p$ is an odd prime, and of order $2^6$.
  In fact there are groups $G$ with these orders such that
  $\abs{\Z(G)}\geq \abs{G:G'}$.
  With GAP, we see also that  all  groups of order $2^5$
  contain anticentral elements.
\end{example}
\begin{example}
  Let $G$ be the group of all upper triangular $n\times n$
  matrices with entries from some finite field of order $q$
  and $1$'s on the main diagonal.
  Then $\abs{G}= q^{\binom{n}{2}}$, $\abs{G:G'} = q^{n-1}$.
  If $a\in G$ has minimal polynomial $(x-1)^n$, then the
  centralizer of $a$ in the matrix ring is the ring generated by
  $a$ and the field. Every element of the centralizer can thus be
  written uniquely as a polynomial in $a$ of degree at most $n-1$
  over the field in question. The centralizer in $G$ contains
  the polynomials with constant term $1$, so
  $\abs{\C_G(a)}= q^{n-1}$.
  Thus such an elements is anticentral.
  This example shows that there exist groups with anticentral
  elements  of arbitrarily large nilpotency index.
\end{example}
\begin{example}
  Let $G$ be an extraspecial group of order $p^5$, and
  $a\in G \setminus \Z(G)$. Then $a$ is anticentral in $G$,
  but not in $\C_G(a)$, since $\C_G(a)$ is nonabelian.
  Thus an anticentral element of a group may not be anticentral
  in a subgroup containing it.
\end{example}
\begin{remark}
  If a group contains anticentral elements, then the restriction
  of every nonlinear character to the commutator subgroup is
  reducible. Every nonlinear character is imprimitive.
\end{remark}
\begin{proof}
  If $\chi \in \irr G$
  with  $\chi_{G'} \in \irr G'$, then
  $\sum_{g\in G'} \abs{\chi(ag)}^2 =
  \abs{G'}$
  for every $a\in G$~\cite[Lemma~8.14]{isa76}.
   But if $a$ is anticentral in $G$,
  then every nonlinear character
  vanishes on $aG'$.

  Now remember that every group is an $M$-group with
  respect to $G'$, that means for every $\chi \in \irr G$ there is
  a subgroup $H$ with
  $G' \leq H \leq G$ and $\psi\in \irr H$ such
  that
  $\chi = \psi^G$ and
  $\psi_{G'}\in \irr G'$~\cite[Theorem~6.22]{isa76}.
  From the first statement, we clearly have $H < G$ for nonlinear
  $\chi$.
\end{proof}
\begin{example}
  A solvable group with anticentral elements need not be monomial:
  Let $G$ be the central product of $\mathbf{SL}(2,3)$ and
  a nonabelian group of  order $8$ with identified centers:
  \[ G = (\mathbf{SL}(2,3) \times E) / Z \text{ where }
         Z = \{ (1,1), (-1, z)\},  1\neq z\in\Z(E).\]
  Then the elements $(x,y)Z$ where $x \in \mathbf{SL}(2,3)$
  has order $3$ and
  $y \in E \setminus \Z(E)$ are anticentral. It is easy to compute
  that the centralizer of such an element is
  $\C(x) \times \C(y)/Z$ and has order $12$.
  This is also the index of the derived subgroup,
  since $G' \cong Q_8$
  and
  $\abs{G}=96$.
  The group is not  monomial, more precisely,
  the irreducible characters of degree $4$ are not monomial.
\end{example}
\begin{example}
 A $p$-group of maximal class contains anticentral elements,
 namely the elements not in
 $\C_G( \mathbf{K}_i(G)/\mathbf{K}_{i+2}(G)) $ for $i\geq2$ where
 $\mathbf{K}_i(G)$ are the terms of the descending central series
 of $G$~\cite[Hilfssatz~III.14.3, Satz~III.14.23]{hupp67}.
  Since here the anticentral
 elements have order $\leq p^2$, the class of nilpotency is not
 bound in terms of the order of the anticentral element.
\end{example}

\section{Supplements}
\begin{prop}\label{gset}\label{suppl}
  Let $a\in G$ be anticentral and $\Omega$ a $G$-set on which $G'$
  acts transitively. Then $a$ fixes exactly one element of
  $\Omega$.
  If $H\leq G$ is a supplement for $G'$ (that is $G=HG'$), then
  $a$ is contained in a unique conjugate of $H$.
  If $a \in H $ and $HG'=G$, then
  \begin{enumthm}
  \item \label{comm}$ [a, H] = H' = H \cap G'$. In particular,
        $a$ is anticentral in $H$.
  \item \label{normg} $\{ x \in G \mid a \in H^x \} = H$.
  \item \label{cent}$\C_G(a) \leq H$.
  \item \label{norm} $\N_G(H) = H$.
  \end{enumthm}
\end{prop}
\begin{proof}
  Let $\omega \in \Omega$.
  By $G'$-transitivity,
  $\omega  = \omega a x$ for some $x\in G'$. Since $a$ is
  anticentral, $x=[a,g]$ for some $g\in G$. Then
  $\omega = \omega a [a,g] = \omega a^g$.
  Thus $\omega g^{-1}$ is fixed by $a$. This shows the existence
  of a fixed element. Uniqueness will be proved later.
  That $a$ is contained in a unique conjugate of a supplement $H$
  will follow  since the action on
  the right cosets of $H$ is $G'$-transitive.
  In (i)-(iv) we can
  assume that $H$ is an element stabilizer of a
  $G'$-transitive $G$-set $\Omega$.
  So for the rest of the proof, choose $\omega\in \Omega$ with
  $\omega a = \omega $ and let $H= G_{\omega}$.
   We have
  \[ [a,H] \subseteq H' \leq H \cap G' \]
  and
  \[ \abs{H\cap G'} = \frac{\abs{H}}{\abs{H:H\cap G'}}
                   = \frac{\abs{H}}{\abs{G:G'}}
                   = \frac{\abs {H}}{\abs{\C_G(a)}}
                   \leq \abs{H : \C_H(a)}
                   = \abs{[a,H]}.\]
  Thus equality holds throughout.
  This proves~\ref{comm} and also~\ref{cent}.
  If $\omega_1$ is another element fixed by $a$, then
  $\omega = \omega_1 g$ for some $g\in G$ by transitivity.
  Then $\omega = \omega_1 g = \omega_1 a g = \omega g^{-1} ag
               = \omega a^{-1} g^{-1} a g = \omega [a,g]$.
  Thus $[a,g] \in H \cap G' = [a,H]$, so
  $[a,g]= [a,h]$  for some $h \in H$.
  Since then $gh^{-1} \in \C_G(a)\leq H$ by~\ref{cent}, we get
  $g\in H$ and thus $\omega_1 = \omega g^{-1}= \omega$. Therefore $a$
  fixes a unique element of $\Omega$. The remaining assertions
  \ref{normg} and \ref{norm} are consequences of this.
\end{proof}
\begin{remark}
  An alternative proof runs as follows:
  Let $\pi$ be the permutation character associated with the action
  of $G$ on $\Omega$.
  Since $G'$ acts transitively on $\Omega$, the only
  linear constituent of $\pi$ is $1_G$, the trivial character, with
  multiplicity $1$. As the nonlinear irreducible characters vanish
  on  $a$, we see $\pi(a)=1$, as was to be  shown.
  \ref{normg} is a consequence of this, from which
  \ref{cent} and \ref{norm} follow. \ref{comm} is also a
  consequence of \ref{normg}: if $[a,g]\in H$, then
  $g\in H$ by~\ref{normg} and so
  $H \cap G' = H \cap [a,G] \subseteq [a,H]$.
\end{remark}
\begin{cor}
  If $G$ contains anticentral elements, then every supplement of
  $G'$ in $G$ is abnormal.
\end{cor}
\begin{proof}
  Let $G=HG'$ and $x\in G$. We have to show that
  $x\in L= \erz{H, H^x}$. Let $a\in G$ be anticentral.
  Then  $a$ is contained in a unique conjugate of $H$ which we can
  assume to be $H$ itself. Now $a$ is anticentral in $L$ and thus
  fixes a unique $L$-conjugate of $H^x$, say $a\in H^{xy}$ where
  $y\in L$.
  From uniqueness in $G$, we get $H=H^{xy}$.
  Thus $xy \in \N_G(H)=H\leq L$ and so $x\in L$, as desired.
\end{proof}
\begin{lem}\label{frat}
 Let $P\in \syl_p(N)$, where $N\nteq G$.
 Then $G= \N_G(P)N'$.
 If $N$ is $\pi$-solvable then the same holds for
 $P$ a Hall $\pi$-subgroup.
\end{lem}
\begin{proof}
 As $PN'$ is characteristic in $N$, it follows that
 $PN' \nteq G$.
 By the Frattini argument applied to $PN'$,
 we have $G= \N_G(P)PN'=\N_G(P)N'$.
\end{proof}
\begin{cor}\label{hall}
  Let $N\nteq G$ and $p$ a prime.
  Then any anticentral element  fixes a unique
  Sylow $p$-subgroup of $N$.
  If $N$ is $\pi$-solvable then the same holds for
  Hall $\pi$-subgroups.
\end{cor}
\begin{proof}
 Let $P\in \syl_p(N)$.
 By Lemma~\ref{frat}, we have $\N_G(P)N'=G$.
 Thus $N'$ and then $G'$ act transitively on the
 Sylow $p$-subgroups of $N$.
 By Proposition~\ref{gset},
 the result follows.
\end{proof}
We can  use this to extend a result of
Gross~\cite[Lemma~3]{rowley95} on groups admitting a fixed point
free automorphism:
\begin{prop}
 If $G$ has an anticentral element $a$, and if
 $N$ has a cyclic Sylow $p$-subgroup, where $N\nt G$ and $N\leq G'$,
 then $N$ has a normal $p$-complement.
\end{prop}
\begin{proof}
 Let $P$ be a cyclic Sylow $p$-subgroup of $N$. We can assume that
 $a\in H = \N_G(P)$. By Proposition~\ref{suppl},
 we have $H' = [a,H]= H \cap G'$. But as $\aut P $ is abelian,
 we see $H' \leq \C_G(P)$. In particular,
 $\N_N(P)=H\cap N \leq H\cap G' \leq \C_G(P)$.
 By Burnside's transfer theorem,
 $N$ has a normal $p$-complement.
\end{proof}
 If $a$ is anticentral and fixes
$P\in \syl_p G$, then $a$ is anticentral in $\N_G(P)$
by Proposition~\ref{suppl}. We now take
a closer look at this situation:
\begin{prop}
 Suppose that $G$ has a normal Sylow $p$-subgroup $P$.
 Let $a= kx = xk$ where $x=a_{p}$ and $k=a_{p'}$, and let $K$ a
 complement to $P$ in $G$ containing $k$.
 Then $a$ is anticentral in $G$ if and only if the following
 conditions hold:
 \begin{enumerate}[label=\textup{(\arabic*)}]
 \item $k$ is anticentral in $K$,
 \item $x$ is anticentral in $\C_P(k)$,
 \item $\C_P(k) \cap P' = \C_P(k)'$,
 \item $\C_P(K)= \C_P(k)$.
 \end{enumerate}
\end{prop}
\begin{proof}
  If $a$ is anticentral in $G$, then $aP= kP$ is anticentral in
  $G/P \cong K$.
  We have
  \[ \abs{G:G'} = \abs{G:K'[P,K]P'}
                = \abs{K:K'}\abs{P:[P,K]P'}.\]
  Since $a$ is anticentral,
  \[ \abs{P:[P,K]P'} = \abs{G:G'}_{p}
      = \abs{\C_G(a)}_{p} = \abs{\C_P(a)}. \]
  As we have coprime action of $K$ on the $p$-group $P$,
  it follows that
  $P= [P,K]\C_P(K)$ and $P/P' = [P,K]P'/P' \times \C_{P/P'}(K)$
  \cite[Theorems~5.3.5 and 5.2.3]{gor68}.
  Let $U= \C_P(k)$. Then
  \begin{align*}
    \abs{\C_U(x)} &= \abs{\C_P(a)} \\
       & = \abs{P:[P,K]P'}  = \abs{\C_{P/P'}(K)}\\
       &= \abs{\C_P(K)P':P'} &&\text{\cite[Satz~I.18.6]{hupp67}}      \\
       &\leq \abs{\C_P(k)P':P'} \tag{*}\\
       &= \abs{UP':P'} = \abs{U:U\cap P'} \\
       &\leq \abs{U:U'} \leq \abs{\C_U(x)}. \tag{**}
  \end{align*}
  Thus equality holds throughout, so in particular
  $x$ is anticentral in $U=\C_P(k)$,
  since $\abs{\C_U(x)}= \abs{U:U'}$.
  Moreover, $U \cap P' = U'$, so condition~(3) holds.
  Finally, we have
  $\C_P(K)P' = \C_P(k)P'$ and thus
  $ U = \C_P(k) \leq \C_P(K)P'$.
  It follows
  \[  U = \C_P(K)P' \cap U = \C_P(K)(P'\cap U)
         = \C_P(K)U'\]
  and we conclude $U=\C_P(K)$.

  Conversely, assume that the four conditions in the statement of
  the proposition hold. Then $x\in \C_P(k) = \C_P(K)$ and thus
  $\C_G(x) = K\C_P(x)$.
  It follows that
  \begin{align*}
    \C_G(a) &= \C_G(k) \cap \C_G(x) = \C_K(k)\C_P(k) \cap K\C_P(x)
    \\
             &= \C_K(k)\big(\C_P(k)\cap \C_P(x)\big)
                = \C_K(k)\C_U(x).
  \end{align*}
  Since $k$ is anticentral in $K$, we have
  $\abs{K:K'}=  \abs{\C_K(k)}$. The other assumptions ensure that
  the inequalities (*) and (**)  become
  equalities, so $  \abs{P:[P,K]P'}= \abs{\C_U(x)}$.
  Then
  \[ \abs{\C_G(a)} = \abs{K:K'}\abs{P:[P,K]P'} = \abs{G:G'}\]
  and the result follows.
\end{proof}
\begin{prop}
 Let $H \leq G$ with  $HG'=G$ and $a\in H$ anticentral in $G$.
 Let
 $P\in \syl_p(G)$ be the Sylow $p$-subgroup of $G$ with
 $P^a=P$, and $S\in \syl_p (H)$ with $S^a=S$.
 Then $P\cap H = S$.
\end{prop}
\begin{proof}
 By Proposition~\ref{suppl}, $a$ is anticentral in $H$ and,
 by Corollary~\ref{hall},
 $a$ fixes
  unique Sylow $p$-subgroups $P$ and $S$ of $G$ and $H$
 respectively.
 We work by induction on $\abs{P}/\abs{S}$.
 If $\abs{P}= \abs{S}$, the assertion is obvious.
 Let $\abs{S}< \abs{P}$. Set $N= \N_H(S)$.
 Then $NH'=H$ and $NG'= NH'G'=HG'=G$.
 Let $M =\N_G(S)$. Now $MG'=G$.
 Again, there is an unique   $Q\in \syl_p(M)$ such that $Q^a=Q$.
 As $S$ is a proper subgroup of some
 Sylow $p$-subgroup of $G$ and $S \nteq M$, we see $S < Q$,
 so by induction $ P\cap M= Q$.
 But then
 \[ S \leq Q\cap N = P \cap M \cap N = P \cap N \leq P\cap H,\]
 which is a $p$-subgroup of $H$. Thus $S = P \cap H$, as desired.
\end{proof}
It is well known that there are nilpotent supplements of $G'$.
Next we show that there is only one nilpotent supplement up to
conjugacy if the group contains anticentral elements.
\begin{notat}
 For $g\in G$, define inductively subsets $\C^i(g)$ by
 \[ \C^0 (g) = 1 \text{ and }
    \C^{i+1}(g) = \{ x \in G \mid [g,x] \in \C^i(g)\}.\]
 Let \[\C^{\infty} (g) = \bigcup_{i\geq 0}\C^i(g).\]
\end{notat}
\begin{thm}\label{minsupp}
 Let $a\in G$ anticentral and set $D=\C^{\infty}(a)$. Then
 \begin{enumthm}
   \item \label{Carter}$D$ is a nilpotent
         selfnormalizing subgroup of $G$,
   \item $G = D G'$ ($D$ is a supplement of $G'$ in $G$),
   \item \label{minimal} Every supplement of $G'$ containing $a$
          contains     $D$,
   \item \label{maximal} Every nilpotent subgroup containing $a$
           is contained in  $D$,
   \item $D$ is the only nilpotent supplement of $G'$
          containing $a$.
 \end{enumthm}
\end{thm}
\begin{proof}
  We begin with~\ref{minimal}:
  Let $H$ be some subgroup containing $a$ with $HG'=G$.
   By induction, assume
  $\C^i(a) \subseteq H$. If $x\in \C^{i+1}(a)$, then
  $[a,x]\in \C^i(a)\subseteq H$ and so $a^x \in H$.
  By Proposition~\ref{suppl}~\ref{normg}, $x\in H$.
  This establishes~\ref{minimal}.

  To show~\ref{maximal}, let $a\in H \leq $G.
  Then the terms of the ascending central series of $H$ are
  contained in the $\C^i(a)$'s and thus $H\subseteq D$,
  if $H$ is nilpotent. 
  \ref{maximal} is proved.
  However, it is well known  (and easy to prove) that there is a
  nilpotent supplement $H$ of $G'$~\cite[Satz III.3.10]{hupp67},
  and  by
  Proposition~\ref{gset}, we can choose $H$ with $a\in H$.
  By \ref{minimal} and \ref{maximal}, $H = D$.
  By Proposition~\ref{suppl}, every supplement is selfnormalizing.
  Now everything follows.
\end{proof}
Recall that  selfnormalizing nilpotent subgroups
are called  Carter subgroups,
so $D$ is a Carter subgroup of $G$.
 In solvable groups,
Carter subgroups always exist and are
conjugate~\cite[Satz~VI.12.2]{hupp67}.
\begin{cor}\label{sylownorm}
 Let $a$ anticentral in $G$ and $N\nteq G$. For every prime $p$,
 $\C^{\infty}(a)$ normalizes a Sylow $p$-subgroup of $N$.
 We have
 \[ \C^{\infty}(a) = \bigcap\{ \N_G(P)  \mid
                            P \in \syl G \text{ and } P^a=P\}.\]
\end{cor}
\begin{proof}
 By Corollary~\ref{hall}, the anticentral element
 $a$ normalizes some
 Sylow $p$-subgroup $P$ of $N$; so $a\in \N_G(P)$.
 By Lemma~\ref{frat}, we have that $G=\N_G(P)N'$,
 and thus surely $G=\N_G(P)G'$.
 By Theorem~\ref{minsupp}\ref{minimal}, it follows that
 $\C^{\infty}(a) \leq \N_G(P)$. This proves the first statement.
 The argument shows that
 $\C^{\infty}(a)$ normalizes the unique Sylow
 $p$-subgroup of $N$ that is normalized by $a$.
 Applying this to the Sylow subgroups of $G$
 normalized by $a$ , we see that
 the intersection of their
 normalizers contains $\C^{\infty}(a)$.
 On the other hand, this intersection is nilpotent and contains
 $a$, and thus by
 Theorem~\ref{minsupp}\ref{maximal}
 is contained in $\C^{\infty}(a)$. The equation of the corollary
 follows.
\end{proof}
Before stating the next corollary, we adopt the following
terminology \cite{doerk92}:
A \emph{Hall system} of a group $G$ is a set of Hall subgroups of
 $G$ containing exactly on Hall $\pi$-subgroup for every
 $\pi \subseteq \{ p \mid p \text{ prime }, p \mid \abs{G}\}$
 and such that for two subgroups $H$ and $K$ in this set we have
 $HK = KH$. A \emph{complement basis} is a set consisting of
 $p$-complements of $G$, exactly one for each prime dividing
 $\abs{G}$.
 A \emph{Sylow basis} is a set consisting of pairwise permutable
 Sylow $p$-subgroups of $G$, exactly one for each prime $p$
 dividing $\abs{G}$. If $G$ is a solvable group, then
 Hall systems, complement bases and Sylow systems exist. Every
 Hall system contains exactly one complement basis and exactly one
 Sylow basis.
 If $K= \{ H_{p'} \mid p \mid \abs{G}\} $ is a complement basis,
 then
 \[ \{ G_{\pi} := \bigcap_{p \notin \pi} H_{p'} \mid
            \pi \subseteq \{p \mid p\mid \abs{G}\}\}\]
 is the unique Hall system containing $K$.
 If $B= \{ S_p \mid p \mid \abs{G}\}$ is a Sylow basis, then
 \[\{ G_{\pi}:= \prod_{p \in \pi} S_p
            \mid \pi \subseteq \{p \mid p \mid \abs{G}\} \} \]
 is the unique Hall system containing
 $B$~\cite[p.~220-222]{doerk92},\cite[Satz~VI.2.2]{hupp67}.
\begin{cor}\label{hallsys}
 Let $N $ be a solvable normal subgroup of $G$ and
 $a\in G$ anticentral.
 Then the set
 \[ \{H\leq N \mid H \text{ a Hall subgroup of } N
          \text{ and } H^a=H \}\]
 is a Hall system of $N$.
 If $G=N$ is solvable, the system normalizer of this Hall
 system  is $\C^{\infty}(a)$.
\end{cor}
\begin{proof}
 The second assertion follows from Corollary~\ref{sylownorm}
  and the first from Corollary~\ref{hall} by taking an
  $a$-invariant  Hall $p'$-group for every prime $p$ and
  then their  intersections.
\end{proof}
\section{Solvability}
We begin with a result characterizing  anticentral elements in
solvable groups in terms of chief factors. It was pointed out to
me by R.~Kn\"{o}rr.
\begin{prop}\label{charaz}
  Let $G$ be solvable. An element $a$ is anticentral in $G$
  if and only if the following conditions hold:
  \begin{enumerate}[label=\textup{(\arabic*)}]
  \item \label{fpf} $a$ operates fixed point freely on the
         noncentral         chief factors of $G$.
  \item If $N/K$ is a central chief factor and $N\leq G'$, then
        $\C_G(aK) < \C_G(aN)$.
  \end{enumerate}
\end{prop}
\begin{proof}
  We begin with ``only if'': 
  Let $N/K$ be a chief factor of $G$.
  Since $aK$ is anticentral in $G/K$, we can assume $K=1$, so $N$
  is a minimal normal subgroup.
  The system normalizer $D= \C^{\infty}(a)$ avoids the noncentral
  chief factors \cite[Hauptsatz~VI.11.10]{hupp67} and contains
  $\C_G(a)$. Thus $a$ operates fixed point freely on $N$ if $N$ is
  noncentral. If $1<N \leq G' = [a,G]$ then obviously
  $\C_G(a) < \C_G(aN) =\{ x \in G \mid [a,x] \in N \}$.

  Now let us prove the other direction: Assume the conditions of
  the proposition and let $1< N \nteq G$ a minimal normal subgroup.
  By induction, $aN$ is anticentral in $G/N$.
  If $N\cap G'=1$, then
  $\C_G(aN) = \{x \in G \mid [a,x]\in N\} = \C_G(a)$ and
  $\C_{G/N}(aN)= \C_G(a)/N$.
  Thus
  \[\abs{G:G'} = \abs{(G/N) : (G/N)'}\abs{N}
    = \abs{\C_{G/N}(aN)}\abs{N}= \abs{\C_G(a)},\]
  so $a$ is anticentral.
  Otherwise we have $N\leq G'$. Then
  \[ \abs{G:G'}= \abs{G/N : (G/N)'} = \abs{\C_{G/N}(aN)}
              = \abs{C_G(aN)}/\abs{N}.\]
  We want to show that this equals $\abs{\C_G(a)}$.
  Since $\abs{\C_G(aN):\C_G(a)} =  \abs{[a,\C_G(aN)]}$
  and $[a,\C_G(aN)]\subseteq N$, we can finish the proof by showing
  $[a,\C_G(aN)]= N$.
  If $N \leq \Z(G)$, then the map $x \mapsto [a,x]$
  is a group homomorphism from $\C_G(aN)$ to $N$ and
  is nontrivial since by assumption
  $\C_G(a) < \C_G(aN)$, and so
  $1\neq [a, \C_G(aN)]\leq N$. Since $\abs{N}$ is a prime,
  we get $[a, \C_G(aN)]= N$ as desired.
  If $N\cap \Z(G)=1$, then $a$ operates fixed point freely on $N$
  and thus $N = [a,N] \subseteq [a, \C_G(aN)] $.
  This finishes the proof.
\end{proof}
Any group admitting a fixed point free automorphism is solvable
\cite{rowley95}
(the proof of this uses the classification
 of  finite simple groups).
So every group for which condition~\ref{fpf} above
holds is necessarily solvable.
We are now going to prove the converse, that is, every group
containing anticentral elements is solvable. In parts, the proof
is very similar to Rowley's proof on the solvability of groups
admitting a fixed point free automorphism \cite{rowley95}.

In the proof, we will need the first part of the following
proposition:
\begin{prop}\label{invcls}
  Let  $a \in G$ be anticentral. If $K$ is a $G$-class which at
  the same time is an $G'$-class, then $a$ fixes exactly one
  element of $K$. Let $D= \C^{\infty}(a)$. The map
  \[  \Z(D) \ni x \mapsto x^G  \]
  is a bijection from $\Z(D)$ to the set of $G'$-classes
  that are invariant under $G$.
\end{prop}
\begin{proof}
  The first statement is just another application of
  Proposition~\ref{gset}.
  Let $x\in \Z(D)$. Then $x^G=x^{DG'}=x^{G'}$.
  Conversely, if $K$ is a conjugacy class of $G$ whose elements
  are conjugate under $G'$, then $K\cap \C_G(a)$ contains exactly
  one element
  $x$, say. We claim $x\in \Z(D)$.
  Assuming this for the moment, we see
  that  the map sending $K$ to the unique element in $K \cap \C_G(a)$
  is  the inverse of the map sending
  $x\in \Z(D)$ to its conjugacy class $K= x^G$.
  So let us prove the claim: As $x^G = x^{G'}$, we have $G = \C_G(x)G'$.
  By Theorem~\ref{minsupp} and since $a\in \C_G(x)$, we have
  $D\subseteq \C_G(x)$ and thus
  $x \in C_G(D) =\Z(D)$, where the last equality follows from
  Proposition~\ref{suppl}.
\end{proof}
\begin{thm}\label{acsolvable}
  Every finite group containing anticentral elements is solvable.
\end{thm}
\begin{proof}
  Let $G$ be a counterexample  with minimal order, and
  $D = \C^{\infty}(a )$ for some anticentral element $a$.
  \begin{step}
    There is $K \nt G$  such
    that $G=DK$ and $K \cong S^m$ for some nonabelian simple group
    $S$.
  \end{step}
  As $G$ is nonsolvable, it contains a nonabelian chief factor,
  say $K/L$. The anticentral elements of $G$ remain anticentral in every
  factor group of $G$, so by minimality $L=1$.
  Since $K$ is characteristically simple, $K\cong S^m$ for some
  simple group $S$.
  By Proposition~\ref{suppl}, $DK$ also contains anticentral
  elements. By minimality, $G=DK$.
  \begin{step}
    If $C$ is a
    characteristic conjugacy class of $S$
    (that is, $C$ is invariant under automorphisms of $S$),
    then $\erz{x} \cap C = \{x\}$ for all $x \in C$.
  \end{step}
  If $C$ is characteristic,  then
  $C^m = C \times \dots \times C$ is a conjugacy class
  of $K$ invariant under all automorphisms of $K$, so in
  particular it is $G$-invariant. This means that $C^m$ is a
  $G$-set on which $K$ acts transitively.
  Then $G'$ also acts transitively on $C^m$.
  By Proposition~\ref{invcls}, $C^m \cap \C_G(a)= \{x\}$
  for some $x$. Thus $C^m \cap \erz{x} = \{x\}$.
  Let $x = (x_1,\dots, x_m)$, where $x_i \in C$.
  Then we see
  $C \cap \erz{x_1} = \{x_1\}$, and the last equation holds for
  all elements of $C$.
  \begin{step}
    $S$ is not isomorphic to an alternating group.
  \end{step}
  Let $C$ be the conjugacy class of
  $x = (1,2,\dots , n-2)(n-1,n)^{n+1} \in A_n$, where $n>4$.
  Then $C$ contains all elements with the same cycle pattern as
  $x$ and $\C_{A_n}(x) = \erz{x}$.
  This conjugacy class is invariant under all
  automorphisms of $A_n$: If $n=6$, it contains all elements of
  order $4$, and if $n\neq 6$, the automorphism group of $A_n$ is
  known to be the $S_n$.
  (One can also show that for any other element
  $y\in A_n$ of order $n-2$,
  one has $\C_{A_n}(y) > \erz{y}$.)
  Since $\abs{C \cap \erz{x}} = \phi(n) \geq 2$, we see from
  Step~2 that $S$ is not an alternating group.
  \begin{step}
   $S$ is not isomorphic to a sporadic simple group.
  \end{step}
  For any sporadic simple group,  there is  $n>2$ such
  that all elements of order $n$ are conjugate~\cite{atlas}.
  Again, we can apply Step~2.
  \begin{step}
    $S$ is not a simple group of Lie type.
  \end{step}
  Assume  $S$ is a simple group of Lie type. If $(B_i,N_i)$,
  $i=1,\dots,m$, are BN-pairs of $S$, then
  $B= B_1 \times \dots \times B_m$ and
  $N= N_1 \times \dots \times N_m$ constitute a BN-pair of
  $K=S^m$. So we can assume that $K$ has a BN-pair $(B,N)$.
  Further, $B$ is
  the normalizer of a Sylow $p$-subgroup $Q$ of $K$.
  By Corollary~\ref{hall}, we can assume that $a$ normalizes $Q$,
  and by Theorem~\ref{minsupp} that $D$ normalizes $Q$.
  So $B = \N_G(Q) \cap K \nt \N_G(Q)= DB$ and $D\cap K \leq B$.
  Let $\Delta $ be the associated building \cite{brown89}.
  Then $\Delta$ as simplicial complex is
  obtained from the poset $\{ Px \mid x \in K, B\leq P \leq K\}$
  under reverse inclusion
  on which $K$ acts by right multiplication.
  This poset is isomorphic with the poset of parabolic subgroups
  $\{ P^x \mid x \in K , B \leq P \leq K\}$ on which $K$ acts by
  conjugation, since each
  $P$ is its own normalizer~\cite[p.111]{brown89}.
  Since $D$ normalizes $B$, also $G=DK$ acts on this set by
  conjugation.
  Thus $G$ acts on
  $\Delta$. Now we can complete the proof exactly as
  Rowley~\cite{rowley95} does: Let $\mathcal{A}$ be a system of
  apartments for $\Delta$.
  Since $\mathcal{A}$ is unique~\cite[p.~93, Theorem~2]{brown89},
  $G$ acts  upon $\mathcal{A}$.
  Since $K$ acts transitively upon the set of pairs $(C,\Sigma)$
  where $C$ is a chamber of $\Delta$, $\Sigma \in \mathcal{A}$
  and $C \in \Sigma$~\cite[p.~112, Theorem]{brown89},
  Proposition~\ref{gset} implies that the
  anticentral element fixes a unique pair $(C, \Sigma)$.
  As $\Sigma $ is a finite Coxeter complex, there is a unique
  chamber $-C$ in $\Sigma$ (the opposite chamber of $C$)
  whose distance from $C$ equals the
  diameter of $\Sigma$. But then the anticentral element also
  fixes $(-C, \Sigma)$, contradiction.

  By the classification of finite simple groups and Steps~3-5,
  the Theorem is proved.
\end{proof}
The reader may ask if we could have applied Step~2 to the simple
groups of Lie type as well. Most of the simple groups seem to have
conjugacy classes   forbidden by Step~2, but there are
counterexamples: In the groups $\mathbf{PSL}(2,3^{2k+1})$, the only
characteristic conjugacy classes are $\{1\}$ and the class of
involutions.

We also remark that the proof of Step~5 shows the following:
a finite group containing a BN-pair can not contain anticentral
elements since it acts transitively on the set of pairs
$(C,\Sigma)$ as above and no group element fixes exactly one pair.

The last result suggests the problem of classifying solvable groups
containing anticentral elements. A first step would be to
characterize $p$-groups containing anticentral elements, but even
this seems to be difficult. I even don't have an idea if there are
more groups with or without anticentral elements of order $p^n$
for $n$ tending to infinity.

\section*{Acknowledgements}
This work is part of my thesis\footnote{Note added in 2023:
In the end, my PhD-thesis was about something else,
and this work was not included in the thesis.}
done under the supervision of 
R.~Kn\"{o}rr. 
I would like to thank him for his many helpful remarks in our 
conversations about this paper, 
and especially for suggesting the subject of ``anticentral'' elements.

\printbibliography
\end{document}